\documentclass{birkjour}
\usepackage{subfigure}
\usepackage{color}
\usepackage{verbatim}
 \newtheorem{thm}{Theorem}[section]
 \newtheorem{corollary}[thm]{Corollary}
 \newtheorem{lemma}[thm]{Lemma}
 \newtheorem{Proposition}[thm]{Proposition}
 \theoremstyle{definition}
 
 \theoremstyle{remark}
 \newtheorem{remark}[thm]{Remark}
 
 \numberwithin{equation}{section}

 \newcommand{\R}{\mathbb{R}}

    \newcommand{\B}{\mathcal{B}}
    \renewcommand{\P}{\mathcal{P}}
\renewcommand{\L}{\mathcal{L}}

\begin{document}

%
%

\title[Curvature lines of an eq\"uiaffine vector field ]
 {{\Large Curvature lines of a transversal eq\"uiaffine vector field along a surface in $3$-space }}

\author[M.Craizer]{Marcos Craizer}

\address{%
Departamento de Matem\'{a}tica- PUC-Rio\br
Rio de Janeiro, RJ, Brasil}
\email{craizer@puc-rio.br}

\author[R.A.Garcia]{Ronaldo Garcia}

\address{%
Instituto de Matem\'atica e Estat\'istica- UFG\br
Goi\^ania, GO, Brasil}
\email{ragarcia@ufg.br}

\thanks{The authors want to thank CNPq and CAPES (Finance Code 001)  for financial support during the preparation of this manuscript. The second author is coordinator of Project PRONEX/ CNPq/ FAPEG 2017 10 26 7000 508. \newline E-mail of the corresponding author: craizer@puc-rio.br}



\date{November 10, 2018}

\begin{abstract}
In this paper we discuss the behavior of the curvature lines of a transversal eq\"uiaffine vector field along a surface in $3$-space at isolated umbilical points. 
\end{abstract}

\maketitle

\section{Introduction}\label{sec:Int}

Let $M$ be a smooth surface and denote by $\mathfrak{X}(M)$ the space of smooth vector fields tangent to $M$. Given an immersion $f:M\to\R^{3}$ and an arbitrary transversal vector field $\xi:M\to\R^3$, write, for $X,Y\in\mathfrak{X}(M)$,
\begin{equation}\label{eq:AffineMetric}
D_Xf_*Y=f_*(\nabla_XY)+h(X,Y)\xi,
\end{equation}
where $\nabla$ is a torsion free connection and $h$ a symmetric bilinear form. We shall assume that $h$ is positive definite, which corresponds geometrically to the convexity of the surface $f(M)$.
For $X\in\mathfrak{X}(M)$, write
\begin{equation}\label{eq:Shape1}
D_X\xi=-f_*(BX)+\tau(X)\xi.
\end{equation}
where $B$ is a $(1,1)$-tensor called the {\it shape operador} and $\tau$ a $1$-form. The vector field $\xi$ is called {\it eq\"uiaffine} if $\tau=0$. 

The class of eq\"uiaffine transversal vector fields includes the Euclidean normal and the Blaschke affine normal vector fields.  
The behavior of curvature lines at an umbilical point
has a vast literature in the Euclidean case (\cite{Bruce1989},\cite{darboux},\cite{GS-1982}) and has also been studied in the affine Blaschke case (\cite{Barajas}). 
Curvature lines of immersions endowed with eq\"uiaffine transversal vector fields
are closely related to asymptotic lines of surfaces in $4$-space (\cite{Craizer-Garcia}). 

When $\xi$ is eq\"uiaffine, $B$ is self-adjoint with respect to the metric $h$, and so admit a pair of linearly independent eigenvectors at each point. 
The lines tangent to these eigenvectors of $B$ are called {\it curvature lines}. 
Points $r_0\in M$ where $B(r_0)$ is a multiple of the identity are called {\it umbilical} and at such points curvature lines are not defined. 

Consider an isolated umbilical point $r_0\in M$. In a fixed neighborhood $V\subset M$ of $r_0$, consider an 
$h$-orthonormal frame $\{X_1,X_2\}$, and denote by $B(r)=(b_{ij}(r))$ the matrix of the shape operator in this frame. Since $B$ is self-adjoint with respect to $h$, we have $b_{12}=b_{21}$. Now consider the 
vector field on $V$ given by
\begin{equation}\label{def:PlanarVectorField}
\B=(b_{11}-b_{22})X_1+2b_{12}X_2.
\end{equation} 
It is easy to verify that $r_0$ is an isolated zero of $\B$ and the index of $\B$ at $r_0$ is twice the index of the curvature lines 
at $r_0$. We say that $r_0$ is {\it semi-homogeneous} of degree $k$ if $\B$ has zero $(k-1)$-jet and $r_0$ is an isolated zero of the $k$-jet of $\B$. 
We verify in section 3 that this definition does not depend on the choice of the $h$-orthonormal frame $\{X_1,X_2\}$.

In this paper we prove a version of Loewner's conjecture, which generalizes the well-known result for euclidean normal vector fields (\cite{Ivanov},\cite{Titus}). 
\begin{thm}\label{thm:Loewner1}
Assume $r_0$ is a semi-homogeneous umbilical point of a pair $(f,\xi)$, where $f:M\to\R^3$ is an immersion with positive definite metric and $\xi$
is an eq\"uiaffine transversal vector field. Then the index of the curvature line foliation at $r_0$ is at most $1$. 
\end{thm}

Natural questions are whether or not the semi-homogeneous or the eq\"uiaffine conditions can be dropped in Theorem \ref{thm:Loewner1}:

\medskip\noindent
{\bf Question 1:} Assume $r_0$ is an umbilical point of an eq\"uiaffine pair $(f,\xi)$, not necessarily semi-homogeneous. Is it still true that the index of the curvature line foliation at $r_0$ is at most $1$?

\medskip\noindent
{\bf Question 2:} Assume $r_0$ is an umbilical point of a pair $(f,\xi)$, not necessarily eq\"uiaffine. Is it still true that the index of the curvature line foliation at $r_0$ is at most $1$?


Loewner's type results always have consequences regarding Carath\'eodory's type conjectures, which states that
a compact ovaloid should have at least $2$ umbilical points. A related question is whether or not there exists an ovaloid with only two Blaschke umbilical points. 
We show in section 5 that 
every compact rotational surface admit at least one umbilical parallel.

The paper is organized as follows: In section 2 we discuss line congruences, which is related but not an essential tool in the proof of Theorem \ref{thm:Loewner1}.
In section 3 we describe the main tools needed in the proof of Theorem \ref{thm:Loewner1}, while in section 4 we prove Theorem \ref{thm:Loewner1}. In section 5 we discuss
the corresponding Carath\'eodory's type results. 


\section{Line Congruences}

The notion of line congruence is not needed for the proof of Theorem \ref{thm:Loewner1}. Nevertheless, it is instructive to 
describe the concepts involved in this theorem in terms of line congruences. 

A $3$-dimensional line congruence is a $2$-dimensional family of lines in $\R^3$. Given a $3$-dimensional line congruence, consider smooth functions $f,\ \xi:M\to\R^3$ such that, for each $r\in M$,
$f(r)$ is a point of the line at $r$ and $\xi(r)$ is a non-zero vector in the direction of the line at $r$. 

\paragraph{Principal lines of a congruence}

Let $r(t)$ be a curve on $M$. Then we set the induced $1$-parameter family of lines by $f(t)=f(r(t))$ and $\xi(t)=\xi(r(t))$. The ruled surface defined by this family is developable
if and only if 
\begin{equation}\label{eq:CurvatureCongruence}
\left[  \xi, f_t,\xi_t \right]=0.
\end{equation}
The solutions of Equation \eqref{eq:CurvatureCongruence} are called the {\it curvature lines} of the congruence.

In a local coordinate system $(u,v)$ for $M$, we can write Equation \eqref{eq:CurvatureCongruence} as
\begin{equation*}
P\frac{du}{dt}^2+2Q\frac{du}{dt}\frac{dv}{dt}+R\frac{dv}{dt}^2=0,
\end{equation*}
where, denoting by $[\cdot,\cdot,\cdot]$ the determinant of $3$ vectors,
\begin{equation*}
P=\left[ \xi, f_u,\xi_u\right],\ \ 2Q=\left[ \xi, f_v,\xi_u\right]+\left[ \xi, f_u,\xi_v\right],\ \ R=\left[ \xi, f_v,\xi_v\right].
\end{equation*}
Writing
\begin{equation*}
\left\{
\begin{array}{c}
\xi_u=b_{11}f_u+b_{21}f_v+\tau_1\xi\\
\xi_v=b_{12}f_u+b_{22}f_v+\tau_2\xi
\end{array}
\right.   \ ,
\end{equation*}
we obtain
$$
P=b_{21},\ \ 2Q=b_{22}-b_{11},\ \ R=-b_{12},
$$
which implies that the curvature lines of the congruence are tangent to the eigenvectors of the shape operator $B$ defined by Equation \eqref{eq:Shape1}.
On the other hand, we remark that the curvature lines of the congruence are independent of the choice of $f$ or $\xi$.

\paragraph{Eq\"uiaffine vector fields}

We say that $\xi$ is eq\"uiaffine with respect to the re\-fe\-rence surface $f$ if $\tau=0$, where $\tau$ is given by Equation \eqref{eq:Shape1}.

\begin{lemma}\label{lemma:CongruenceExact}
There exists a vector field $\tilde\xi=\lambda\xi$ such that the pair $(f,\tilde\xi)$ is eq\"uiaffine if and only if $\tau$ is exact.
\end{lemma}

\begin{proof}
Assume that $\tau=d\mu$ and take $\tilde{\xi}=\exp{(-\mu)}\xi$. Then
$$
\tilde{\xi}_*X=\exp{(-\mu)}(\xi_*X-d\mu(X)\xi)=-f_*(\exp{(-\mu)}BX),
$$
thus proving that $\tilde\xi$ is eq\"uiaffine. Conversely, assume that $\tilde\xi=\exp{(-\mu)}\xi$ is eq\"uiaffine. Then the above calculations 
show that $\tau=d\mu$, thus proving that $\tau$ is exact.
\end{proof}

We say that the pair $(f,\xi)$ is {\it exact } when $\tau$ is exact. We have the following question:

\medskip\noindent
{\bf Question 3:}
Given a line congruence $\left(f,\xi\right)$, is there a reference surface $\tilde{f}=f+\lambda\xi$ such that the pair $( \tilde{f},\xi )$ is exact?

\section{Main Tools}
 
In this section we describe the main tools needed in the proof of Theorem \ref{thm:Loewner1}.

\subsection{Isothermal coordinates}

Consider an immersion $f:M\to\R^3$ together with an eq\"uiaffine transversal vector field $\xi$, and let $h$ denote the affine metric defined by Equation \eqref{eq:AffineMetric}.
We say that $(u,v)$ are isothermal coordinates on $M$ if 
$$
h_{11}=h_{22}=\rho;\ h_{12}=0.
$$
The isothermal property is independent of the choice of $\xi$ and it is well-known that any convex surface can be covered by isothermal parameterizations.
Let $\nu$ denote the co-normal vector field defined by 
\begin{equation} \label{eq:CoNormal}
\nu\cdot f_u=\nu\cdot f_v=0,\ \ \nu\cdot\xi=1.
\end{equation}

\begin{lemma}
We can write 
\begin{equation}\label{eq:Iso1}
\xi=\frac{1}{[\nu,\nu_u,\nu_v]}\nu_u\times\nu_v.
\end{equation}
Moreover,
\begin{equation}\label{eq:Iso2}
\nu_u\cdot f_u=\nu_v\cdot f_v=-\rho, \ \ \nu_u\cdot f_v=\nu_v\cdot f_u=0,
\end{equation}
and
\begin{equation}\label{eq:Iso3}
f_u=\frac{\rho}{[\nu,\nu_u,\nu_v]} \nu\times\nu_v,\ \ f_v=-\frac{\rho}{[\nu,\nu_u,\nu_v]} \nu\times\nu_u.
\end{equation}
Finally
\begin{equation}\label{eq:Iso4}
[f_u,f_v,\xi]=\frac{\rho^2}{[\nu,\nu_u,\nu_v]}.
\end{equation}
\end{lemma}
\begin{proof}
From $\xi$ eq\"uiaffine and $\nu\cdot\xi=1$, we conclude \eqref{eq:Iso1}. From $h_{11}=\rho$ we obtain $\nu\cdot f_{uu}=\rho$ and so $\nu_u\cdot f_u=-\rho$. 
With a similar reasoning, we conclude \eqref{eq:Iso2}. Observe that $\nu_v\cdot f_u=0$ and $\nu\cdot f_u=0$ to obtain the first of Equations \eqref{eq:Iso3}. The second one is obtained similarly. Equation \eqref{eq:Iso4} is obtained from Equations \eqref{eq:Iso3}.
\end{proof}

\begin{remark}
In case $\xi$ is Blaschke, $[\nu,\nu_u,\nu_v]=[f_u,f_v,\xi]=\rho$ and the above formulas become much simpler (see \cite[N4, p.208]{Nomizu}).
\end{remark}

\begin{lemma}\label{lemma:ShapeOperator}
The shape operator is given by
\begin{equation*}
B=-\frac{1}{\delta}
\left[
\begin{array}{cc}
\nu_{uu}\cdot\xi & \nu_{uv}\cdot\xi\\
\nu_{uv}\cdot\xi & \nu_{vv}\cdot\xi
\end{array}
\right],
\end{equation*}
where 
\begin{equation}\label{eq:Definedelta}
\delta=\frac{[f_u,f_v,\xi][\nu,\nu_u,\nu_v]}{\rho}.
\end{equation}
\end{lemma}

\begin{proof}
Writing $\xi_u=-b_{11} f_u-b_{21} f_v$ we obtain
$$
-b_{11}[f_u,f_v,\xi]=[\xi_u,f_v,\xi]=\frac{\rho}{[\nu,\nu_u,\nu_v]^2}[\nu_u,\nu_v,\nu_{uu}]=\frac{\rho}{[\nu,\nu_u,\nu_v]}\nu_{uu}\cdot\xi,
$$
which proves the formula for $b_{11}$. The other formulas are obtained similarly.
\end{proof}

\subsection{The vector field $\mathcal{B}$}\label{sec:Statements}

For an isolated umbilical point $r_0$,
let $B=(b_{ij})$ be the matrix of the shape operator in an $h$-orthonormal frame $\{X_1,X_2\}$ and 
consider the vector field $\B$ defined by Equation \eqref{def:PlanarVectorField}.

The {\it order} of the umbilical point is the order of the first non-zero jet of $\B$. We call {\it simple} the umbilical points of order $1$. 
We say that $r_0$ is {\it semi-homogeneous} of degree $k$ if it is an isolated zero of the first non-zero jet $J_k\B$ of $\B$.

\begin{lemma}\label{lemma:Independent}
The above definition does not depend on the choice of the $h$-orthonormal frame $\{X_1,X_2\}$. 
\end{lemma}
\begin{proof}
If $\{\tilde{X}_1,\tilde{X}_2\}$ is another $h$-orthonormal frame, we can write 
\begin{equation*}
\left\{
\begin{array}{c}
\tilde{X}_1=\cos(\theta)X_1-\sin(\theta)X_2\\
\tilde{X}_2=\sin(\theta)X_1+\cos(\theta)X_2.
\end{array}
\right.
\end{equation*}
Straightforward calculations show that 
\begin{equation*}
\left\{
\begin{array}{c}
\tilde{b}_{11}-\tilde{b}_{22}=\cos(2\theta)(b_{11}-b_{22})-\sin(2\theta)2b_{12}\\
2\tilde{b}_{12}=\sin(2\theta)(b_{11}-b_{22})+\cos(2\theta)2b_{12},
\end{array}
\right.
\end{equation*}
and the same relation holds for the $k$-jet of $\mathcal{B}$. Thus $r_0$ is isolated for the $k$-jet of $\mathcal{B}$ if and only if it is isolated for the $k$-jet of $\tilde{\mathcal{B}}$.
\end{proof}

Consider now isothermal coordinates $(u,v)$. Then we can write $f_u=\rho X_1$, $f_v=\rho X_2$, for some $h$-orthonormal frame $\{X_1,X_2\}$. Since
$\xi_u=\rho D_{X_1}\xi$, $\xi_v=\rho D_{X_2}\xi$, we conclude that the matrix of the shape operator $B$ in coordinates $(u,v)$ is the same as in the frame
$\{X_1,X_2\}$. Thus we can use Lemma \ref{lemma:Independent} also for the coordinate lines of an isothermal parameterization. 

From now on we shall always be using an isothermal parameterization. In particular, $b_{12}=b_{21}$. The following lemma is well-known and relate
the index of the curvature lines with the index of $\B$ at an umbilical point.


\begin{lemma}\label{lemma:IndexCurvLines}
The index of the curvature lines of $\xi$ at an umbilical point $r_0=(u_0,v_0)$ is exactly one half of the index of the vector field $\B$. 
\end{lemma}

\begin{proof}
The equation of curvature lines is given by 
$$
(b_{11}-b_{22}) dudv+b_{12}(dv^2-du^2)=0.
$$
Writing $du=dr\cos(\eta)$, and $dv=dr\sin(\eta)$, we obtain 
$$
(b_{11}-b_{22}) \sin(2\eta)- 2b_{12}\cos(2\eta)=0,
$$
which implies 
that $\B$ is a multiple of $(\cos(2\eta),\sin(2\eta))$. Thus the rotation of $\B$ is twice the rotation $[du:dv]$, which proves the lemma.
\end{proof}


\subsection{The support function and the vector field $\P$}

Consider
the {\it support function} $p:M\to\R$ defined by
\begin{equation*}
p(u,v)=\nu(u,v)\cdot (f(u,v)-q_0),
\end{equation*}
where $q_0\in\R^3$ is a fixed point. Observe that 
\begin{equation*}\label{eq:Dg}
p_u=\nu_u\cdot (f-q_0),\ p_v=\nu_v\cdot(f-q_0).
\end{equation*}
Differentiating we obtain
\begin{equation*}\label{eq:D2g}
p_{uu}=\nu_{uu}\cdot (f-q_0)-\rho;\ p_{vv}=\nu_{vv}\cdot (f-q_0)-\rho;\ p_{uv}=\nu_{uv}\cdot(f-q_0).
\end{equation*}
It follows that
\begin{equation}\label{eq:D2ga}
p_{uu}-p_{vv}=(\nu_{uu}-\nu_{vv})\cdot (f-q_0);\ 2p_{uv}=2\nu_{uv}\cdot(f-q_0).
\end{equation}
Denote
\begin{equation}\label{eq:DefineP}
P_1=p_{uu}-p_{vv}, \ \ P_2=2p_{uv}.
\end{equation}
and let $\P$ be the the vector field whose component are $P_1$ and $P_2$.

\subsection{Loewner's conjecture for planar vector fields} 

Consider a function $w:U\subset\R^2\to\R$ defined in a neighborhood $U$ of $(0,0)$ and assume $(0,0)$ is an isolated zero 
of the vector field $W=\frac{\partial^n w}{\partial {\bar z}^n}$. Loewner's conjecture states that the index
of $W$ is at most $n$. This conjecture is proved in \cite{Titus}, but there are also some controversies concerning 
the proof. Nevertheless, assuming that $W$ is semi-homogeneous, i.e., $(0,0)$ is an isolated zero of its first non-zero jet, then 
there are no doubts that Loewner's conjecture holds (\cite{Ivanov},\cite{Xavier1},\cite{Titus},\cite{Xavier2}). 

Thus we can state the following theorem, considering $n=2$:

\begin{thm}\label{thm:LoewnerGeneral}
Let $w:U\subset\R^2\to\R$ be a function defined in a neighborhood $U$ of $(0,0)$ and assume $(0,0)$ is an isolated zero 
of the semi-homogeneous vector field $W=(w_{uu}-w_{vv},2w_{uv})$. 
Then the index of $W$ is at most $2$. 
\end{thm}

\section{Proof of Theorem \ref{thm:Loewner1}}

The idea to prove Theorem \ref{thm:Loewner1} is to apply the Theorem \ref{thm:LoewnerGeneral} to $\P$ defined by Equation \eqref{eq:DefineP} and conclude results concerning $\B$.  

\subsection{Simple umbilical points}

From now on we shall fix an isolated umbilical point $(u_0,v_0)=(0,0)$ in $M$. Denote $f_0=f(0,0)$, $\xi_0=\xi(0,0)$ and
let $q_0=f_0+\lambda_0^{-1}\xi_0$, where 
$$
\lambda_0=b_{11}(0,0)=b_{22}(0,0).
$$
Denote also $\delta_0=\delta(0,0)$, where $\delta$ is defined by Equation \eqref{eq:Definedelta}.  

From Equations \eqref{eq:D2ga} and \eqref{eq:DefineP}, we obtain 
\begin{equation*}
\P(0,0)=-\lambda_0^{-1}\left(  (\nu_{uu}-\nu_{vv})\cdot\xi_0, 2\nu_{uv}\cdot\xi_0 \right).
\end{equation*}
Now Lemma \ref{lemma:ShapeOperator} implies that 
\begin{equation*}
\P(0,0)=\lambda_0^{-1}\delta_0 \B(0,0).
\end{equation*}
Since $\B(0,0)=0$, we conclude that $\P(0,0)=0$.

\begin{lemma}
We have that
\begin{equation}\label{eq:Umbilic1}
f(u,v)+\lambda_0^{-1}\xi(u,v)=f_0+\lambda_0^{-1}\xi_0+O(2),
\end{equation}
where $O(2)$ denotes terms of degree at least $2$ in $(u,v)$. 
Conversely, if equation \eqref{eq:Umbilic1} is satisfied, then $(0,0)$ is an umbilical point and $\lambda_0=b_{11}(0,0)=b_{22}(0,0)$.
\end{lemma}

\begin{proof}
The equation $f(u,v)+\lambda_0^{-1}\xi(u,v)=f_0+\lambda_0^{-1}\xi_0$ holds up to order $1$ if and only if 
$$
f_u+\lambda_0^{-1}\xi_u=f_v+\lambda_0^{-1}\xi_v=0
$$
at $(0,0)$, which is equivalent to say that $(0,0)$ is an umbilical point and $\lambda_0=b_{11}(0,0)=b_{22}(0,0)$. 
\end{proof}

\begin{Proposition}\label{Prop:jet1P}
Let $(0,0)$ be an umbilical point and $q_0=f_0+\lambda_0^{-1}\xi_0$. Then 
\begin{equation*}
J_1\P(0,0)=\lambda_0^{-1}\delta_0 J_1\B(0,0),
\end{equation*}
where $J_1\P$ and $J_1\B$ denote the first jet of the vector fields $\P$ and $\B$, respectively.  
\end{Proposition}

\begin{proof}
We shall verify the equality for the derivatives of $P_i$ and $B_i$ with respect to $u$ and $v$. We shall verify the equality
for the derivatives of $P_1$ and $B_1$ with respect to $u$, the other $3$ cases being similar. Observe that 
$$
(P_1)_u=p_{uuu}-p_{vvu}=(\nu_{uuu}-\nu_{vvu})\cdot (f-q_0)+(\nu_{uu}-\nu_{vv})\cdot f_u.
$$
At $(0,0)$, $f_u=-\lambda_0^{-1}\xi_u$ and so
$$
(P_1)_{u}=-\lambda_0^{-1} \left( (\nu_{uuu}-\nu_{vvu})\cdot\xi +(\nu_{uu}-\nu_{vv})\cdot \xi_u \right).
$$
On the other hand, since $B_1=-\frac{1}{\delta}(\nu_{uu}-\nu_{vv})\cdot\xi$, we have that
$$
\delta (B_1)_u+\delta_uB_1=-(\nu_{uuu}-\nu_{vvu})\cdot\xi -(\nu_{uu}-\nu_{vv})\cdot \xi_u.
$$
Thus, at $(0,0)$, 
$$
\delta_0 (B_1)_u=\lambda_0(P_1)_u,
$$
thus proving the desired result.
\end{proof}

\subsection{Umbilical points of order $k$}

\begin{lemma}\label{lemma:UmbK}
Let $(0,0)$ be umbilical and $q_0=f_0+\lambda_0^{-1}\xi_0$. Then $(0,0)$ is umbilical of order $\geq k$ if and only if 
\begin{equation}\label{eq:Umbilick}
f(u,v)+\lambda_0^{-1}\xi(u,v)=q_0+O(k+1).
\end{equation}
\end{lemma}

\begin{proof}
Write
\begin{equation}\label{eq:Shape}
\begin{array}{c}
\xi_u=-b_{11}f_u-b_{21}f_v\cr
\xi_v=-b_{12}f_u-b_{22}f_v
\end{array}
\end{equation}

If the $(k-1)$-jet of $\B$ at $(0,0)$ equals $0$, differentiating equation \eqref{eq:Shape} $(k-1)$ times we obtain that, at $(u_0,v_0)$,
the $k$-jet of $\xi$ equals $-\lambda_0$ times the $k$-jet of $f$,
thus proving formula \eqref{eq:Umbilick}.

Conversely, assume that equation \eqref{eq:Umbilick} holds with $k\geq 2$. Differentiating Equation \eqref{eq:Shape} and taking
$(u,v)=(0,0)$, we obtain 
$$
(b_{11})_uf_{u}+(b_{21})_uf_v=0, \ \ (b_{11})_vf_u+(b_{21})_vf_v=0,
$$
which implies that $(b_{11})_u=(b_{21})_u=(b_{11})_v=(b_{21})_v=0$. Similarly, we can prove that 
$(b_{12})_u=(b_{22})_u=(b_{12})_v=(b_{22})_v=0$, which implies that the $1$-jet of $\B$ vanishes at $(0,0)$. 
To prove that the $(k-1)$-jet of $\B$ vanishes at $(0,0)$, one can proceed by induction.
\end{proof}

\begin{Proposition}\label{prop:1}
At an umbilical point of order $\geq k$, 
\begin{equation}\label{eq:PkAk}
J_k\P= \lambda_0^{-1}\delta_0 J_k\B,
\end{equation}
where $J_k\P$ and $J_k\B$ denote the $k$-jets of $\P$ and $\B$, respectively.
\end{Proposition}

\begin{proof} 
For $k=2$, we must prove that the second derivatives of $\P$ and $\B$
are multiples at $(0,0)$. Let us consider $(P_1)_{uu}$ and $(B_1)_{uu}$, the other cases being similar.
Observe that 
$$
(P_1)_{uu}=(\nu_{uuuu}-\nu_{vvuu})\cdot (f-q_0)+2(\nu_{uuu}-\nu_{vvu})\cdot f_u+(\nu_{uu}-\nu_{vv})\cdot f_{uu}.
$$
At $(u_0,v_0)$, by Lemma \ref{lemma:UmbK},
$$
(P_1)_{uu}=-\lambda_0^{-1}\left( (\nu_{uuuu}-\nu_{vvuu})\cdot\xi+2(\nu_{uuu}-\nu_{vvu})\cdot\xi_u+(\nu_{uu}-\nu_{vv})\cdot \xi_{uu}\right).
$$
On the other hand, 
$$
\delta (B_1)_{uu}+2\delta_u(B_1)_u+\delta_{uu}B_1=-(\nu_{uuuu}-\nu_{vvuu})\cdot\xi-2(\nu_{uuu}-\nu_{vvu})\cdot \xi_u-(\nu_{uu}-\nu_{vv})\cdot\xi_{uu}.
$$
At $(0,0)$, 
$$
\delta_0 (B_1)_{uu}=\lambda_0(P_1)_{uu},
$$
thus proving the claim. To prove Equation \eqref{eq:PkAk} for any $k$, one can proceed by induction.
\end{proof}

\subsection{Completing the proof}

\begin{lemma}\label{lemma:1}
If $(0,0)$ is an umbilical point, semi-homogeneous of degree $k$, the index of $\B$ at $(0,0)$ is the same index of $J_k\B$
at $(0,0)$. 
\end{lemma}

\begin{proof}
For $||(u,v)||$ sufficiently small, 
$$
||\B-J_k\B|| \leq \frac{1}{2}||J_k\B||,
$$
which proves the lemma.
\end{proof}

We can now prove Theorem \ref{thm:Loewner1}. 

\begin{proof}
If $(0,0)$ is an isolated umbilical point, semi-homogeneous of degree $k$, Proposition \ref{prop:1} implies that the index of $J_k\B$ at $(u_0,v_0)$ is the same index of $J_k\P$.
From the above lemma, they are also equal to the index of $\B$. 
By Theorem \ref{thm:LoewnerGeneral}, any vector field of the form $\P_k$ has index $\leq 2$. By Lemma \ref{lemma:IndexCurvLines}, 
the index of the curvature lines is at most $1$.
\end{proof}

\section{Carath\'eodory's type Results and Questions}

\subsection{A Carath\'eodory's type result}

In the Euclidean case, Carath\'eodory's conjecture states that any compact surface homeomorphic to the sphere admits
at least $2$ umbilical points. This conjecture is a consequence of Loewner's conjecture, which states that the index
of the Euclidean curvature lines at umbilical points are at most $1$. 

In our case, Loewner's type Theorem \ref{thm:Loewner1} implies the following result:

\begin{corollary}\label{cor:Cara}
Consider a convex centroaffine immersion $f:M^2\to\R^3$, $M^2$ compact, and an eq\"uiaffine transversal vector field $\xi$. Assume that all umbilical
points are semi-homogeneous. Then there are at least $2$ umbilical points.
\end{corollary}

If Question $1$ in the Introduction is answered positively, then we can also drop the semi-homogeneous condition in Corollary \ref{cor:Cara}.

\subsection{Rotational surfaces}

Concerning Corollary \ref{cor:Cara}, it is natural to ask whether or not there exist 
compact surfaces in $3$-space with only two umbilical points. If we are free to choose the eq\"uiaffine transversal vector field $\xi$, then 
the answer is positive, since we can take $\xi$ to be the Euclidean unit normal and the surface to be a rotational ellipsoid. But what 
happens if we choose $\xi$ to be the Blaschke vector field? Are there compact surfaces with only two Blaschke umbilical points?

To answer this question, it is natural to look at the rotational surfaces. But if we consider rotational surfaces with the Blaschke transversal vector field, 
then there exists at least one umbilical parallel. In this section we prove this surprising fact, and
so the following question remains open:

\medskip\noindent
{\bf Question 4:} Is there an ovaloid in $3$-space with $2$ umbilical points with respect to the Blaschke transversal vector field?

\subsubsection{Shape operator of a rotation surface at the axis}

Consider a surface graph of a function of the form
$$
z=f(x^2+y^2).
$$
Then, by an affine change of variables,  the $4$-jet of $z$ at $(0,0)$ is given by
$$
z=\frac{1}{2} (x^2+y^2)+ \frac{\alpha}{24} \left( x^2+y^2   \right)^2 +O(6),
$$
which implies that $(0,0)$ is umbilical with respect to the Blaschke affine normal. We conclude that the axes points of a rotation surface are Blaschke umbilical.

\subsubsection{Shape operator of a rotation surface outside the axis}

Assume $S$ is a smooth rotation surface generated by the convex planar arc $\gamma(t)=(x(t),y(t))$, $t\in[0,\pi]$, such that $x(t)>0$ and $y'(t)>0$, 
for $t\in(0,\pi)$. A parameterization of $S$ is given by 
\begin{equation*} 
\psi(t,\theta)=\left( x(t)\cos\theta,x(t)\sin\theta, y(t)   \right).
\end{equation*}

To calculate the Blaschke shape operator, we calculate first the co-normal vector field  $\nu$ and the Blaschke metric $h=(h_{ij})$. 
Straightforward calculations show that  
$$
\nu(t,\theta)=\frac{x(t)}{\phi(t)} \left(  -y'(t)\cos\theta, -y'(t)\sin\theta, x'(t) \right),
$$
$$
h_{11}=\frac{x(t)}{\phi(t)},\ \ h_{12}=h_{21}=0,\ \ h_{22}=\frac{x(t)^2y'(t)}{\phi(t)},
$$
where $\phi^4(t)=x(t)^3y'(t)$.

\begin{lemma}
The Blaschke affine normal vector field $\xi$ can be written as
$$
\xi=\left( a\cos\theta, a\sin\theta, b \right),
$$
for certain $a=a(t)$, $b=b(t)$. 
\end{lemma}

\begin{proof}
Observe that $\nu_{\theta}\cdot\xi=0$. The conditions $\nu\cdot\xi=1$ and $\nu_s\cdot\xi=0$ are given by 
\[
\left\{
\begin{array}{l}
-ay'+bx'=\frac{\phi}{x} \cr
 -ay''+bx''=-\left(\frac{\phi}{x}\right)^2 \left(\frac{x}{\phi} \right)' \ \ .
\end{array}
\right.
\]
This system certainly has a solution $\xi=(a,b)$. 
\end{proof}

Write $g=\frac{\phi}{x}$. Then 
\[
\left\{
\begin{array}{l}
-ay'+bx'=g \cr
 -ay''+bx''=g' \ \ .
\end{array}
\right.
\]
We conclude that $a'y'-b'x'=0$, which implies that $\xi_t=\frac{b'}{y'}\psi_t$. Since $\xi_{\theta}=\frac{a}{x}\psi_{\theta}$, the parallels and
meridians are affine curvature lines. Moreover, a parallel is umbilical if and only if $\frac{a}{x}=\frac{b'}{y'}$. 

\subsubsection{A useful parameterization}

Any generator arc $\gamma(t)=(x(t),y(t))$ as above can be parameterized with $y'(t)=x(t)$. In fact, 
take 
$$
t=\int \frac{y'(s)}{x(s)}ds.
$$

This parameterization simplifies a lot the calculations. In fact, assuming this parameterization, we obtain $g=1$ and a parallel is umbilical if and only if $a=b'$. 
Moreover
$$
\xi=(a,b)=\frac{1}{\Delta}\left( x''(t), y''(t)  \right),
$$
where $\Delta=[\gamma',\gamma'']=(y'')^2-y'y'''$ which is positive, by hypothesis. Then a parallel $t$ is umbilical if and only if 
$$
\frac{y'''}{\Delta}=\frac{y'''}{\Delta} -\frac{\Delta'y''}{\Delta^2}.
$$
Thus if $y''(t)=0$, the parallel $t$ is umbilical. 
Since $y''(0)=x'(0)>0$ and $y''(\pi)=x'(\pi)<0$, there exists at least one $t_0\in(0,\pi)$ such that $y''(t_0)=0$. 
We conclude that there exists at least one umbilical parallel $t_0\in(0,\pi)$.

\end{document}